\documentclass[reqno,11pt]{amsart}
\usepackage[left=1in,right=1in,top=1in,bottom=1in]{geometry}
\usepackage{amsmath}
\usepackage{amssymb}
\usepackage{xcolor}
\usepackage{hyperref}
\hypersetup{
   colorlinks=true,
   citecolor=blue,
   filecolor=blue,
   linkcolor=blue,
   urlcolor=blue
}

\newcommand{\al}{\alpha}

\newcommand{\ud}{\mathrm{d}}
\newcommand{\A}{\mathbf{A}}
\newcommand{\ol}{\overline}
\newcommand{\I}{\mathbf{I}}
\newcommand{\Z}{\mathbb{Z}}
\newcommand{\M}{\mathcal{M}}
\newcommand{\cN}{\mathcal{N}}
\newcommand{\Nb}{\mathcal{N}_{\mathrm b}}
\newcommand{\PP}{\mathbf{P}}
\newcommand{\Q}{\mathbf{Q}}
\newcommand{\R}{\mathbb{R}}
\newcommand{\RR}{\mathbf{R}}
\newcommand{\fS}{\mathbf{S}}
\newcommand{\T}{\mathrm{T}}
\DeclareMathOperator{\dist}{dist}
\DeclareMathOperator{\tr}{tr}
\DeclareMathOperator{\diag}{diag}

\DeclareMathOperator{\vol}{vol}

\def\<{\langle}\def\>{\rangle}
\numberwithin{equation}{section}
\theoremstyle{plain}
\newtheorem{thm}{Theorem}[section]

\newtheorem{lem}[thm]{Lemma}
\newtheorem{prop}[thm]{Proposition}

\theoremstyle{definition}
\newtheorem{defn}[thm]{Definition}

\theoremstyle{remark}
\newtheorem{rem}[thm]{Remark}

\title[No point defects in biaxial Landau--de Gennes models]{On the absence of point defects in biaxial Landau--de Gennes models}

\author{Haotong Fu}
\author{Huaijie Wang}
\author{Wei Wang}
\author{Zhifei Zhang}
\address{School of Mathematical Sciences, Peking University, Beijing 100871, China}
\email[Haotong Fu]{2301110012@pku.edu.cn}
\email[Huaijie Wang]{huaijie\_wang@163.com}
\email[Wei Wang]{wwmath166@outlook.com,\,2201110024@stu.pku.edu.cn}
\email[Zhifei Zhang]{zfzhang@math.pku.edu.cn}

\newcommand{\IF}{\mathcal I}
\newcommand{\Spts}{\mathcal S_{\mathrm{pts}}}
\DeclareMathOperator{\Area}{Area}

\begin{document}
\raggedbottom

\begin{abstract}
We prove that every locally energy-minimizing harmonic map from a
three-dimensional domain into the biaxial vacuum manifold of the sextic
Landau--de Gennes potential considered here is smooth in the interior.  This applies to
small-elastic-constant limits of local minimizers with uniform energy and
$L^\infty$ bounds.  The Frobenius metric on the vacuum manifold lifts to a
rescaled Berger metric on its universal cover $\mathbb S^3$.  A hypothetical
point singularity would therefore produce a harmonic two-sphere in this
Berger sphere whose index form satisfies a shifted stability inequality
with constant $1/4$.  We construct a Hopf-adapted variation field that is
smooth across all branch points, has squared norm at most $13/4$, and has
index at most $-8\pi$.  These bounds contradict the shifted stability
inequality.  The resulting absence of interior point defects contrasts
with the point singularities that occur in the uniaxial theory.

\end{abstract}

\maketitle
\setlength{\abovedisplayskip}{4pt plus 1pt minus 1pt}
\setlength{\belowdisplayskip}{4pt plus 1pt minus 1pt}
\setlength{\belowdisplayshortskip}{2pt plus 1pt minus 1pt}

\section{Setting and statement}\label{sec:setting-statement}

\subsection{Introduction}

Nematic liquid crystals are anisotropic fluids whose local orientational order
is described in the Landau--de Gennes theory by a symmetric traceless
$\Q$-tensor. Let $\mathbb{S}_0$ be the Euclidean space of traceless symmetric
$3\times3$ matrices with Frobenius inner product
\[
\langle \A,\mathbf B\rangle_F:=\tr(\A\mathbf B),
 \quad|\A|_F^2:=\tr(\A^2).
\]
The Landau--de Gennes energy consists of an elastic term and a bulk term.  The
elastic energy measures spatial distortions of $\Q$, whereas the bulk potential
selects the energetically preferred phase. For a domain $U\subset\R^3$ and a parameter $\varepsilon\in(0,1)$, the Landau--de Gennes energy with the sextic bulk potential studied by Wang--Zhang
\cite{WZ24} is
\[
E_\varepsilon(\Q;U):=\int_U\left(\frac{1}{2}|\nabla\Q|_F^2+\frac{1}{\varepsilon^2}f_{\mathrm{bulk}}(\Q)\right)\ud x,
\]
where the bulk density is
\begin{equation}
f_{\mathrm{bulk}}(\Q):=a_1-\frac{a_2}{2}\tr(\Q^2)+\frac{a_4}{4}\bigl(\tr(\Q^2)\bigr)^2+\frac{a_6}{6}\bigl(\tr(\Q^2)\bigr)^3+\frac{a'_6}{6}\bigl(\tr(\Q^3)\bigr)^2,\label{sexticbulk}
\end{equation}
and $a_2,a_4,a_6,a'_6>0$, while $a_1$ is chosen so that
\[
\min_{\Q\in\mathbb S_0}f_{\mathrm{bulk}}(\Q)=0.
\]
We are concerned with the small-elastic-constant limit
$\varepsilon\to0^+$, in which the bulk term forces $\Q$ toward the zero set of
$f_{\mathrm{bulk}}$, called the vacuum manifold.  In the standard quartic
model, this zero set is a uniaxial vacuum manifold.  For a material constant
$s_0>0$, it is given by
\begin{equation}
\mathcal N_{\mathrm u}:=\left\{s_0\left(\mathbf n\otimes\mathbf n-\frac{1}{3}\I_3\right):\mathbf n\in\mathbb S^2\right\}
 \cong\mathbb S^2/\{\mathbf n\sim-\mathbf n\}\cong\mathbb R\mathbb P^2.\label{UniaxialVacuumMfd}
\end{equation}
We use $\cong$ for a diffeomorphism of smooth manifolds and, when the
metrics are specified, for an isometry.

The small-elastic-constant limit in the uniaxial setting has been studied
extensively. Majumdar--Zarnescu \cite{MZ10} analyzed the Oseen--Frank limit,
in which point defects arise, while Canevari \cite{Can17} studied
three-dimensional line defects.  Related results include
\cite{Can15,FH22,FWW25,FWW26,NZ13}. The vacuum structure may change for a
general sixth-order bulk density. Sextic potentials can instead produce a
biaxial vacuum manifold \cite{AL08}, and experimental evidence for biaxial
nematic phases has also been reported \cite{SS04}.  A biaxial $\Q$-tensor has
three distinct eigenvalues and therefore determines an orthogonal eigenframe,
up to the natural sign symmetries, rather than the single unoriented director
associated with \eqref{UniaxialVacuumMfd}.

Precisely, the vacuum manifold corresponding to the potential \eqref{sexticbulk} is
\[
 \Nb:=\{\RR\mathbf D\RR^{\T}:\RR\in \mathrm{SO}(3)\}\subset\mathbb{S}_0,
\]
where 
\[
\mathbf D:=r_*\diag(1,-1,0),
\]
and $r_*>0$ is a constant determined by $a_2$, $a_4$, $a_6$, and $a_6'$.
Thus, any element of $\Nb$ has three distinct eigenvalues
$r_*,-r_*$, and $0$.  Choosing an oriented eigenframe identifies its
rotational degrees of freedom with
$\mathrm{SO}(3)/(\mathbb Z_2\times\mathbb Z_2)$, equivalently with
$\mathbb S^3/Q_8$, where $Q_8$ is the quaternion group.

Let $\Omega\subset\R^3$ be a bounded domain.  We write
$U\subset\!\subset\Omega$ when $U$ is open and its closure is a compact
subset of $\Omega$.  For such a set $U$, define
\begin{equation}
 E(\Q;U):=\frac{1}{2}\int_U|\nabla \Q|_F^2\,\ud x,
 \quad
 \Q\in H^1(U;\Nb).\label{EQU}
\end{equation}
Here $H^1(U;\Nb)$ consists of the maps in $H^1(U;\mathbb S_0)$ whose
values lie in $\Nb$ almost everywhere.
We next recall the bounded-energy result of Wang--Zhang \cite{WZ24}.  Suppose
that $\{\Q_\varepsilon\}_{\varepsilon\in(0,1)}\subset
H^1(\Omega;\mathbb{S}_0)$ is a family of local minimizers of
$E_\varepsilon$.  This means comparison on every ball
$B\subset\!\subset\Omega$ with all $H^1(B;\mathbb S_0)$ competitors
having the same boundary trace.  Suppose, for some $M>0$, that
\[
 E_\varepsilon(\Q_\varepsilon;\Omega)\leq M,
 \quad
 \|\Q_\varepsilon\|_{L^\infty(\Omega)}\leq M.
\]
By \cite[Theorem~1.4]{WZ24}, there exist a
sequence $\varepsilon_n\to 0^+$ and a map
$\Q_0\in H^1(\Omega;\Nb)$ such that
\begin{gather*}
 \Q_{\varepsilon_n}\to\Q_0\text{ strongly in }H^1_{\mathrm{loc}}(\Omega;\mathbb S_0),\\
 \frac{1}{\varepsilon_n^2}f_{\mathrm{bulk}}(\Q_{\varepsilon_n})\to0
 \quad\text{in }L^1_{\mathrm{loc}}(\Omega).
\end{gather*}
Moreover, $\Q_0$ is a local energy minimizer of \eqref{EQU} on
$\Omega$.  More precisely, $\Q_0\in H^1_{\mathrm{loc}}(\Omega;\Nb)$ and
satisfies
\[
E(\Q_0;B)\leq E(\Q_{\mathrm c};B)
\]
for any ball $B\subset\!\subset\Omega$ and any
$\Q_{\mathrm c}\in H^1(B;\Nb)$ satisfying $\Q_{\mathrm c}=\Q_0$ on
$\partial B$ in the trace sense.  In the standard harmonic-map terminology
\cite{SU82}, the regular and singular sets of $\Q_0$ are
\begin{align*}
\operatorname{Reg}(\Q_0)&:=\{x\in\Omega:\exists r>0\text{ such that }\Q_0\in C^0(B_r(x);\Nb)\},\\
\operatorname{Sing}(\Q_0)&:=\Omega\backslash\operatorname{Reg}(\Q_0).
\end{align*}
It follows from \cite[Theorems~I--II]{SU82} that
the set $\operatorname{Sing}(\Q_0)$ is closed and, in dimension three,
discrete.
Wang--Zhang \cite[Theorem~1.4(4)]{WZ24} also proved that, away from
$\operatorname{Sing}(\Q_0)$, for any integer $j\geq0$,
\[
\Q_{\varepsilon_n}\to\Q_0\text{ in }
C^j_{\mathrm{loc}}\bigl(\Omega\backslash\operatorname{Sing}(\Q_0);\mathbb S_0\bigr).
\]
We write $\Spts:=\operatorname{Sing}(\Q_0)$ and refer to it as the point-defect
set. The purpose of this paper is to show that $\Spts$ is empty in the biaxial
setting.  Our main result is the following.

\begin{thm}[Interior regularity]\label{thm:main}
Assume that $\Q_0\in H_{\mathrm{loc}}^1(\Omega;\Nb)$ is a local minimizer of
\eqref{EQU}.  Then $\Q_0\in C_{\mathrm{loc}}^{\infty}(\Omega;\Nb)$ and, in
particular, $\Spts=\emptyset$.
\end{thm}

The geometric reason for this regularity is a quantitative obstruction to
minimizing tangent cones.  Although the lifted target is not round, its
Hopf symmetry determines a variation of every nonconstant harmonic
two-sphere that is too unstable to be the link of a minimizing cone.  The
proof compares this target-dependent instability with the universal radial
constant $1/4$; the comparison does not require a global lift of $\Q_0$.

\begin{rem}\label{rem:comments}
The heuristic for Theorem~\ref{thm:main} comes from the universal covering
space of $\Nb$, which is $\mathbb S^3$.  When
$\mathbb S^3$ is equipped with the round metric induced by its standard
embedding in $\mathbb R^4$, the regularity results of Schoen--Uhlenbeck
\cite[Theorem~2.7]{SU84} imply that three-dimensional energy-minimizing harmonic maps into
$\mathbb S^3$ do not have singularities.  Lin--Wang \cite[Theorem~1]{LW06} proved the corresponding conclusion for
stable stationary harmonic maps.  These results suggest the absence of
point defects in the present biaxial problem.

The round-sphere result does not prove
Theorem~\ref{thm:main}, because one cannot simply lift $\Q_0$ through the
covering map and apply \cite{LW06,SU84}.  The Frobenius
metric on $\Nb$ lifts to a metric on $\mathbb S^3$ that differs from the round
metric induced by standard embedding $\mathbb S^3\subset\mathbb R^4$.
Krantz \cite[Theorem~1.1]{arXiv:2605.03809} proves a related
codimension-four regularity result for certain compact simple Lie groups
with bi-invariant metrics.  That result does not apply to the present
Berger metric.  The lifted target metric is therefore an essential part of
the regularity problem.

The absence of point defects does not exclude line defects.  In the
logarithmic-energy regime, nontrivial boundary data for global minimizers can
force line defects whose structure is governed by the induced homotopy class.
See Canevari--Fu--Wang \cite{CFW26} for a three-dimensional framework that
includes uniaxial and biaxial vacuum manifolds.
By contrast, point defects occur in the uniaxial theory.  The radial
hedgehog
\[
 \Q_{\mathrm H}(x)
 :=s_0\left(
   \frac{x}{|x|}\otimes\frac{x}{|x|}-\frac{1}{3}\I_3
 \right),
 \quad x\in\R^3\backslash\{0\},
\]
is the standard point-defect profile and is singular at the origin.  Theorem
\ref{thm:main} shows that no analogous interior point singularity can occur for
locally minimizing maps into the biaxial vacuum manifold $\Nb$.
\end{rem}

\subsection{Proof strategy}

As explained in Remark~\ref{rem:comments}, the universal cover of $\Nb$ is
topologically $\mathbb S^3$, but the lifted Frobenius metric is not the round
metric. Consequently, the standard test vector fields used for the round
sphere in the regularity arguments of \cite{LW06,SU84} do not apply
directly here.  The first step of our argument is therefore to compute the
pullback of the Frobenius metric under the
quaternionic covering map and show that, after an explicit diffeomorphism and 
constant rescaling, it is the Berger metric $g_4$.  See
\eqref{eq:g4-pullback-isometry} and \cite{TU12}.

We then argue by contradiction.  Standard blow-up analysis \cite{SU82} shows
that a point singularity would produce a nonconstant zero-homogeneous
minimizing tangent map.  Its link is a harmonic
two-sphere in $\Nb$, which lifts to a harmonic map into the Berger sphere
$(\mathbb S^3,g_4)$.  The minimizing property of the tangent cone yields a
shifted stability inequality for this lifted link.  We construct a vector
field along the lifted harmonic sphere that is adapted to the Berger metric
and differs from the round-sphere fields used in \cite{LW06,SU84}.  This field
combines the global normal of the branched minimal immersion with a tangential
correction determined by the
Hopf Killing field.  The resulting index estimate is incompatible with the
shifted stability inequality and therefore excludes any nonconstant
blow-up limit.

The quantitative comparison is the decisive point.  For the lifted link
$\widetilde\phi$, cone minimality gives
\[
 \IF_{\widetilde\phi}^{g_4}(V,V)
 +\frac{1}{4}\int_{\mathbb S^2}|V|_{g_4}^2\,\ud A_{g_{\mathrm r}}\geq0
 \quad\text{for every smooth section }V.
\]
Here $\IF^{g_4}$ is the harmonic-map index form defined in
\eqref{eq:index-form-definition}.  Proposition
\ref{prop:direct-quantitative-instability} constructs a section $W$ with
$\IF_{\widetilde\phi}^{g_4}(W,W)\leq-8\pi$ and
$|W|_{g_4}^2\leq13/4$.  Since the unit source sphere has area $4\pi$, the
left-hand side for this section is at most $-19\pi/4$.
The main technical task is to preserve the smoothness of $W$ at branch
points while obtaining an energy-index estimate from the Hopf Killing field and the surface identities.

\subsection{Organization of this paper}\label{subsec:organization}

The paper is organized as follows.  Section~\ref{sec:frobenius-pullback}
computes the Frobenius pullback metric and identifies it with a rescaled Berger
metric.  Section~\ref{sec:tangent-cones} derives the shifted stability
inequality satisfied by the link of a hypothetical tangent cone.
Section~\ref{sec:quantitative-instability} constructs the destabilizing vector
field and proves the quantitative instability estimate for harmonic
two-spheres.  Section~\ref{sec:cone-contradiction} combines these results to
contradict the cone stability and prove Theorem~\ref{thm:main}.

\subsection{Notation and conventions}\label{subsec:notation}

We close this section by specifying the geometric and algebraic notation used
throughout the proof.

\begin{itemize}
\item Smooth manifolds are denoted by calligraphic letters.  A Riemannian
manifold is written $(\M^m,h)$ or $(\cN^n,g)$, and $T\M$ is its tangent bundle.
For a smooth vector bundle $E\to\M$, the notation $C^\infty(\M,E)$ denotes
the space of smooth sections $s:\M\to E$. Thus, $s(x)\in E_x$ for any
$x\in\M$.  In particular, a smooth vector field is denoted by
$X\in C^\infty(\M,T\M)$.  For a smooth map $f:\M\to\cN$, a smooth vector
field along $f$ is an element $W\in C^\infty(\M,f^*T\cN)$, which means that
$W(x)\in T_{f(x)}\cN$ for any $x\in\M$.  A local $h$-orthonormal frame
$(e_1,\ldots,e_m)$ on $U\subset\M$ consists of smooth vector fields that
satisfy $h(e_a,e_b)=\delta_{ab}$.
\item The Levi--Civita connection of $(\cN^n,g)$ is denoted by $\nabla^g$.
For a smooth map $f:\M\to\cN$, the pullback bundle $f^*T\cN$ carries the
target metric evaluated at $f(x)$ and the connection $\nabla^{f^*g}$
along $f$.  For a source curve $x(t)$, this connection differentiates
$W(x(t))$ covariantly along the target curve $f(x(t))$.
On source tangent vectors the induced tensor is
$(f^*g)(X,Y)=g(\ud f(X),\ud f(Y))$; it can degenerate when $\ud f$
loses rank.  Our curvature
convention is
\[
 R^g(X,Y)Z
 :=\nabla_X^g\nabla_Y^gZ-\nabla_Y^g\nabla_X^gZ-\nabla_{[X,Y]}^gZ.
\]
The Riemannian volume measure is $\ud\vol_g$.  If $\dim\cN=2$, set
\[
 \ud A_g:=\ud\vol_g,
 \quad
 \Area_g(U):=\int_U\ud A_g.
\]
For $Y\in C^\infty(\cN,T\cN)$, its metric dual is
$Y^\flat:=g(Y,\cdot)\in C^\infty(\cN,T^*\cN)$.

\item The open geodesic ball is
$B_r^g(x):=\{y\in\M:\dist_g(x,y)<r\}$.  We
write $B_r(x)$ when the metric is clear.  In $\R^n$ this always denotes the
Euclidean ball.
\item The matrices and $\Q$-tensors are denoted by capital letters.
For a field $\mathbb F$, $M_k(\mathbb F)$ denotes the algebra of
$k\times k$ matrices with entries in $\mathbb F$.  The symbols $\A^{\T}$ and
$\I_k$ denote transpose and the $k\times k$ identity matrix, respectively.
For arbitrary real matrices, we use
\[
 \langle\A,\mathbf B\rangle_F:=\tr(\A^{\T}\mathbf B),
 \quad |\A|_F^2:=\langle\A,\A\rangle_F.
\]
On $\mathbb S_0$ this agrees with the convention in the introduction.
\item The sphere $\mathbb{S}^k\subset\R^{k+1}$ carries the unit round metric unless a
different metric is specified.  The area measure on the unit round
two-sphere is $\ud A_{g_{\mathrm r}}$.
\item In complex-coordinate formulas, $\mathrm{i}$ denotes the scalar
imaginary unit.  In quaternion formulas, the same symbol denotes the
quaternion basis element.  The context makes the meaning clear.
\end{itemize}

\section{The Frobenius pullback metric}\label{sec:frobenius-pullback}

This section constructs the quaternionic covering of $\Nb$ and identifies the
lifted Frobenius metric with a rescaled Berger metric.  We also collect the
harmonic-map conventions used in the stability argument.
We denote by $g_F$ the Riemannian metric on $\Nb$ induced by the Frobenius
inner product on $\mathbb S_0$.

\subsection{The quaternionic cover and its metric}
\label{subsec:vacuum-cover}

We use unit quaternions to compute the metric on the universal cover.
Let $\mathbb H$ be the quaternion algebra with basis
$(1,\mathrm i,\mathrm j,\mathrm k)$ and
$\mathrm i^2=\mathrm j^2=\mathrm k^2=\mathrm i\mathrm j\mathrm k=-1$.
Quaternionic conjugation is denoted by $q\mapsto\overline q$, and
$\mathbb S^3=\{q\in\mathbb H:q\overline q=1\}$.
Identify $\operatorname{Im}\mathbb H$ with $\mathbb R^3$ using the ordered
basis $(\mathrm i,\mathrm j,\mathrm k)$, and define
\[
 \RR_qv:=qv\overline q,\quad
 \A_av:=a\times v
 \quad(q\in\mathbb S^3,\ a,v\in\operatorname{Im}\mathbb H).
\]
Here $\times$ is the Euclidean cross product.  For
$a=x\mathrm i+y\mathrm j+z\mathrm k$,
\[
 \A_a=\begin{pmatrix}0&-z&y\\z&0&-x\\-y&x&0\end{pmatrix}.
\]

The rotation map and its derivative give the covering used below.
\begin{lem}\label{lem:quaternionic-double-cover}
The map $q\mapsto\RR_q$ is a surjective homomorphism
$\mathbb S^3\to\mathrm{SO}(3)$ with kernel $\{\pm1\}$, and
\begin{equation}
 \left.\frac{\ud}{\ud t}\right|_{t=0}\RR_{\exp(ta)}=2\A_a.
 \label{eq:dR-at-identity}
\end{equation}
\end{lem}
\begin{proof}
We compute the differential and use connectedness to identify the image.
Conjugation preserves the imaginary subspace and its norm.  Since
$\mathbb S^3$ is connected and $\RR_1=\I_3$, it follows that
$\RR_q\in\mathrm{SO}(3)$.  Moreover, $\RR_{pq}=\RR_p\RR_q$.
The quaternion identity $av-va=2a\times v$ gives
\eqref{eq:dR-at-identity}.  Thus the image is an open subgroup of the
connected group $\mathrm{SO}(3)$ and is the whole group.  A unit quaternion
commutes with every imaginary quaternion exactly when it is $1$ or $-1$;
this identifies the kernel and the double cover.
\end{proof}

Define the vacuum covering map and the symmetric matrix basis by
\[
 \PP(q):=\RR_q\mathbf D\RR_q^{\T},\quad
 Q_8:=\{\pm1,\pm\mathrm i,\pm\mathrm j,\pm\mathrm k\},
\]
\[
 \fS_{ab}:=e_a\otimes e_b+e_b\otimes e_a,
 \quad |\fS_{ab}|_F^2=2\quad(1\leq a<b\leq3),
\]
where $(e_1,e_2,e_3)$ is the standard basis of $\mathbb R^3$.
The next lemma gives the deck group and the differential that
selects the distinguished metric direction.

\begin{lem}\label{lem:vacuum-covering}
The map $\PP:\mathbb S^3\to\Nb$ is an eight-sheet covering with fibers
\[
 \PP^{-1}(\PP(q))=qQ_8.
\]
The maps $q\mapsto qh$, $h\in Q_8$, leave $\PP(q)$ fixed; they are the
deck transformations.  Thus
$\mathbb S^3/Q_8\cong\Nb$.  At the identity,
\begin{equation}
 \ud\PP_1(\mathrm i)=-2r_*\fS_{23},\quad
 \ud\PP_1(\mathrm j)=-2r_*\fS_{13},\quad
 \ud\PP_1(\mathrm k)=4r_*\fS_{12}.
 \label{eq:dP}
\end{equation}
\end{lem}
\begin{proof}
We identify the fibers from the distinct eigenvalues of $\mathbf D$, then
check that the differential has full rank.  The stabilizer of $\mathbf D$ is
\[
 K_{\mathbf D}
 =\{\I_3,\diag(1,-1,-1),\diag(-1,1,-1),\diag(-1,-1,1)\}.
\]
Its inverse image under the quaternionic rotation map is $Q_8$, so
\[
 \PP(q')=\PP(q)
 \quad\Leftrightarrow\quad
 q^{-1}q'\in Q_8.
\]
For the differential, \eqref{eq:dR-at-identity} gives
\[
 \ud\PP_1(a)=2[\A_a,\mathbf D].
\]
With $\mathbf D=r_*\diag(1,-1,0)$, direct multiplication gives
\eqref{eq:dP}.  These three images are orthogonal and nonzero.
The equivariance
\[
 \PP(pq)=\RR_p\PP(q)\RR_p^{\T}
\]
therefore gives full rank at every point.  The map is surjective by the
orbit definition of $\Nb$ and proper by compactness of $\mathbb S^3$;
hence it is a covering.  The fiber calculation gives eight sheets and the
stated quotient.  
\end{proof}

The pullback metric measures quaternionic variations through their tensor
variations.  For $q\in\mathbb S^3$ and $U,V\in T_q\mathbb S^3$,
\[
 (\PP^*g_F)_q(U,V)
 =\langle\ud\PP_q(U),\ud\PP_q(V)\rangle_F.
\]
A smooth lift and its tensor-valued projection have the same Dirichlet
energy when the cover carries this metric.

\subsection{Identification with a Berger metric}
\label{subsec:berger-identification}

We next introduce the Hopf field and the Berger metric that appears in the
pullback of the Frobenius metric.

\begin{defn}\label{def:hopf-field}
The standard diagonal circle action on $\mathbb{S}^3$ is
\[
\Phi_t(X):=(X_1\cos t-X_2\sin t,
 X_1\sin t+X_2\cos t,X_3\cos t-X_4\sin t,
 X_3\sin t+X_4\cos t),
\]
where $t\in\R/(2\pi\Z)$.  Its oriented orbits $t\mapsto\Phi_t(X)$ are the
Hopf fibers.  The round Hopf field is the infinitesimal generator of this
circle action:
\begin{equation}
 V_{\mathrm H}\in C^\infty(\mathbb{S}^3,T\mathbb{S}^3),
 \quad
 V_{\mathrm H}(X):=\left.\frac{\ud}{\ud t}\right|_{t=0}\Phi_t(X)
 =(-X_2,X_1,-X_4,X_3).
 \label{eq:round-Hopf-field}
\end{equation}
It satisfies
\[
 V_{\mathrm H}(X)\cdot X=0,
 \quad
 |V_{\mathrm H}(X)|_{g_{\mathrm r}}^2
 =X_1^2+X_2^2+X_3^2+X_4^2=1.
\]
\end{defn}

The pullback computation below singles out the Hopf direction.  We therefore
introduce a Berger metric that weights this direction by a factor of four
relative to the round metric.

\begin{defn}\label{def:berger-metric}
For each $p\in\mathbb{S}^3$, define the bilinear form
$(g_4)_p:T_p\mathbb S^3\times T_p\mathbb S^3\to\mathbb R$ by
\begin{equation}
 (g_4)_p(U,W)
 :=(g_{\mathrm r})_p(U,W)
 +3(g_{\mathrm r})_p(U,V_{\mathrm H})
    (g_{\mathrm r})_p(W,V_{\mathrm H}).
 \label{eq:Berger-convention}
\end{equation}
This defines the Berger metric $g_4$ on $\mathbb S^3$.
\end{defn}

The next lemma identifies the lifted Frobenius metric with $g_4$ up to an
explicit diffeomorphism and a constant factor.

\begin{lem}\label{lem:frobenius-berger-isometry}
There is an explicit diffeomorphism $\Theta:\mathbb{S}^3\to\mathbb{S}^3$ such that
\begin{equation}
 \PP^*g_F=\Theta^*(8r_*^2g_4).
 \label{eq:g4-pullback-isometry}
\end{equation}
Consequently,
\begin{equation}
 (\mathbb{S}^3,\PP^*g_F)
 \cong(\mathbb{S}^3,8r_*^2g_4).
 \label{eq:g4-identification}
\end{equation}
\end{lem}

\begin{proof}
We compare the metrics at the identity.  Write $L_p(q):=pq$ and
$R_p(q):=qp$ for left and right translation.  A left-invariant metric is
preserved by all $L_p$, and a right-invariant metric by all $R_p$; each is
determined by its value at the identity.  Inversion exchanges the two
kinds of invariance.

\smallskip\noindent\textit{The left-invariant pullback.}
For $q\in\mathbb{S}^3$, let
\[
 E_1(q):=q\mathrm{i},
 \quad
 E_2(q):=q\mathrm{j},
 \quad
 E_3(q):=q\mathrm{k}.
\]
This is a round-orthonormal left-invariant frame.  Let $(\sigma_1,\sigma_2,\sigma_3)$ be its dual one-forms, so
$\sigma_i(E_j)=\delta_{ij}$; we write
$\sigma_i^2=\sigma_i\otimes\sigma_i$.
The differential values in \eqref{eq:dP} give
\begin{equation}
 \frac{1}{8r_*^2}\PP^*g_F
 =\sigma_1^2+\sigma_2^2+4\sigma_3^2.
 \label{eq:left-Berger-form}
\end{equation}

\smallskip\noindent\textit{Transport to the right-invariant Berger metric.}
Under the unit-quaternion identification, \eqref{eq:round-Hopf-field} is
\[
 V_{\mathrm H}(q)=\mathrm{i}q.
\]
Thus, $V_{\mathrm H}$ is right invariant, and so is $g_4$.  At $1\in\mathbb{S}^3$,
\[
 (g_4)_1(a,a)
 =|a|^2+3\langle a,\mathrm{i}\rangle^2,
 \quad a\in\operatorname{Im}\mathbb H.
\]

Set
\[
 u:=\frac{1+\mathrm{j}}{\sqrt2},
 \quad
 C_u(q):=uqu^{-1},
 \quad
 \iota(q):=q^{-1},
 \quad
 \Theta:=\iota\circ C_u.
\]
Then $C_u$ is an inner automorphism of
the unit-quaternion group $\mathbb{S}^3$, identified with $SU(2)$.
The map $\iota$ is group inversion, and
\[
 u\mathrm{k}u^{-1}=\mathrm{i},
 \quad
 \ud\Theta_1(a)=-uau^{-1}.
\]
Since $\operatorname{Ad}_u:a\mapsto uau^{-1}$ is orthogonal,
\begin{align}
 (\Theta^*g_4)_1(a,a)
 &=|a|^2+3\langle\operatorname{Ad}_u a,\mathrm{i}\rangle^2=|a|^2+3\langle a,\mathrm{k}\rangle^2.
 \label{eq:theta-metric-at-identity}
\end{align}
Hence \eqref{eq:left-Berger-form} and $\Theta^*g_4$ agree at the identity.

Moreover, $\Theta$ is an anti-homomorphism:
\[
 \Theta(pq)=\Theta(q)\Theta(p).
\]
The anti-homomorphism identity gives, for every $p\in\mathbb S^3$,
\[
 \Theta\circ L_p=R_{\Theta(p)}\circ\Theta.
\]
The right invariance of $g_4$ implies that $\Theta^*g_4$ is left invariant.
Both sides of
\[
 \frac{1}{8r_*^2}\PP^*g_F=\Theta^*g_4
\]
are consequently left invariant and agree at the identity.  This proves
\eqref{eq:g4-pullback-isometry} and \eqref{eq:g4-identification}.
\end{proof}

A Killing field is a vector field whose flow preserves the metric.
The Hopf rotations provide such a field; the next lemma gives its unit
normalization and the derivative used in the surface calculations.

\begin{lem}\label{lem:berger-hopf-field}
Set
\begin{equation}
 \xi:=\frac{1}{2}V_{\mathrm H}\in C^\infty(\mathbb S^3,T\mathbb S^3),
 \quad
 \xi^\flat:=g_4(\xi,\cdot)
 \in C^\infty(\mathbb S^3,T^*\mathbb S^3).
 \label{eq:g4-unit-Hopf-field}
\end{equation}
Orient $(\mathbb S^3,g_4)$ by declaring the right-invariant
$g_4$-orthonormal frame
\[
 \bigl(\xi(q),K(q),J(q)\bigr),
 \quad K(q):=\mathrm{k}q,
 \quad J(q):=\mathrm{j}q,
\]
to be positive.  Let $\times_{g_4}$ denote the corresponding metric cross
product, characterized by
\[
 g_4(U\mathbin{\times_{g_4}}V,W)
 =\operatorname{vol}_{g_4}(U,V,W)
 \quad\text{for }q\in\mathbb S^3\text{ and }U,V,W\in T_q\mathbb S^3.
\]
Then $\xi$ is a unit Killing field, and
\begin{equation}
 \nabla_X^{g_4}\xi=2X\mathbin{\times_{g_4}}\xi.
 \label{eq:bundle-curvature-two}
\end{equation}
\end{lem}

\begin{proof}
We verify the Killing and length properties, then compute the connection
in the oriented invariant frame.

The Hopf circle action preserves both $g_{\mathrm r}$ and $V_{\mathrm H}$,
so it preserves $g_4$ and $V_{\mathrm H}$ is a $g_4$-Killing field.  By
\eqref{eq:Berger-convention},
\[
 |V_{\mathrm H}|_{g_4}^2=4,
\]
and hence $\xi$ is a unit Killing field. For right-invariant vector fields,
the Koszul formula is
\[
 2g_4(\nabla_X^{g_4}Y,Z)
 =g_4([X,Y],Z)-g_4([Y,Z],X)+g_4([Z,X],Y).
\]
Quaternion multiplication gives
\[
 [\xi,J]=-K,
 \quad [\xi,K]=J,
 \quad [J,K]=-4\xi.
\]
Substitution yields
$\nabla_J^{g_4}\xi=2K$, $\nabla_K^{g_4}\xi=-2J$, and
$\nabla_\xi^{g_4}\xi=0$.  The definition of $\times_{g_4}$ then gives
\eqref{eq:bundle-curvature-two}.
\end{proof}

\subsection{Harmonic maps and second variation}
\label{subsec:harmonic-map-conventions}

We finish the section by defining the tension field and index form and 
stating their invariance properties.

\begin{defn}
\label{def:tension-field}
Let $f:(\M^m,h)\to(\cN^n,g)$ be smooth.  The target Levi--Civita connection
induces the pullback connection $\nabla^{f^*g}$ on $f^*T\cN$.  The tension
field is
\[
 \tau_{h,g}(f)
 :=\operatorname{tr}_h(\nabla\ud f)
 =\sum_{\al=1}^m
 \left(
  \nabla^{f^*g}_{e_\al}\bigl(\ud f(e_\al)\bigr)
  -\ud f(\nabla^h_{e_\al}e_\al)
 \right),
\]
where $(e_\al)$ is a local $h$-orthonormal frame.  The map $f$ is a
harmonic map if $\tau_{h,g}(f)=0$, equivalently, if it is critical for the
Dirichlet energy under any compactly supported smooth variation.
\end{defn}

The next definition specifies the second-variation convention used throughout
the proof.

\begin{defn}
Let $f:(\M^m,h)\to(\mathbb S^3,g_4)$ be a harmonic map and let
$W\in C_c^\infty(\M,f^*T\mathbb S^3)$.  Its index form is
\begin{equation}
 \begin{aligned}
 \IF_{f,h}^{g_4}(W,W)
 &:={\left.\frac{\ud^2}{\ud t^2}\right|}_{t=0}
 \left(\frac{1}{2}\int_\M|\ud f_t|_{h,g_4}^2\,\ud\vol_h\right)\\
 &=\int_\M\left(
   |\nabla^{f^*g_4}W|_{h,g_4}^2
   -\sum_{\al=1}^m
    g_4\left(
     R^{g_4}(W,\ud f(e_\al))\ud f(e_\al),W
    \right)
  \right)\,\ud\vol_h,
 \end{aligned}
 \label{eq:index-form-definition}
\end{equation}
where $f_t$ is any compactly supported variation with initial field $W$.
For a noncompact source, the energy derivative is taken on a relatively
compact open set containing the support of the variation; the resulting
index is independent of that choice.
When the source metric is clear, we write
$\IF_f^{g_4}:=\IF_{f,h}^{g_4}$.  We always include the target-metric
superscript.
\end{defn}

The index form is the energy Hessian: a negative value gives a direction
that lowers the energy to second order.  The map is stable when this form
is nonnegative for every allowed variation field.  For smooth compactly supported fields $U,V$ along $f$, its symmetric
bilinear form is
\begin{equation}
 \IF_{f,h}^{g_4}(U,V)
 :=\frac{1}{2}\Bigl(
  \IF_{f,h}^{g_4}(U+V,U+V)
  -\IF_{f,h}^{g_4}(U,U)-\IF_{f,h}^{g_4}(V,V)\Bigr).
 \label{eq:bilinear-index}
\end{equation}
On a compact source without boundary, a smooth field $K$ along $f$ is an energy Jacobi field
when $\IF_{f,h}^{g_4}(V,K)=0$ for every smooth field $V$ along $f$.
This is the kernel condition for the energy Hessian.

The following elementary invariances allow us to normalize both the source and
the target metrics later.

\begin{lem}\label{lem:index-form-invariances}
If $\dim\M=2$ and $\widehat h=e^{2\omega}h$, then the harmonicity and the
index form are unchanged by this conformal change of the source metric:
\begin{equation}
 \tau_{\widehat h,g_4}(f)=e^{-2\omega}\tau_{h,g_4}(f),
 \quad
 \IF_{f,\widehat h}^{g_4}(W,W)=\IF_{f,h}^{g_4}(W,W).
 \label{eq:source-conformal-invariance}
\end{equation}
If $g_4$ is replaced by $cg_4$, $c>0$, then
\begin{equation}
 \IF_{f,h}^{cg_4}(W,W)=c\IF_{f,h}^{g_4}(W,W),
 \quad
 \int_\M|W|_{cg_4}^2\,\ud\vol_h
 =c\int_\M|W|_{g_4}^2\,\ud\vol_h.
 \label{eq:homothety}
\end{equation}
\end{lem}

\begin{proof}
We compare the scale of each trace term with the scale of the source
measure and the target inner product.

In each of the two trace terms in \eqref{eq:index-form-definition}, replacing
an $h$-orthonormal frame by a $\widehat h$-orthonormal frame contributes to the
factor $e^{-2\omega}$, whereas $\ud A_{\widehat h}=e^{2\omega}\ud A_h$.  This proves \eqref{eq:source-conformal-invariance}.  A constant rescaling of the target metric leaves its Levi--Civita connection unchanged and multiplies each
target inner product by $c$, which proves \eqref{eq:homothety}.
\end{proof}

The target rescaling therefore allows us to work with $(\mathbb{S}^3,g_4)$.

\section{Tangent cones and the shifted stability inequality}
\label{sec:tangent-cones}

We derive the stability inequality satisfied by the spherical link of a
hypothetical point-defect tangent cone.

Suppose, for a contradiction, that $x_0\in\Spts$.  Write
\[
 \Q_{\rho}(x):=\Q_0(x_0+\rho x).
\]
By
\cite[Theorem~III and Proposition~4.7]{SU82}, a sequence $\Q_{\rho_j}$ with $\rho_j\to0$ converges strongly in
$H^1_{\mathrm{loc}}$ to a nonconstant zero-homogeneous harmonic map
\[
 u_\infty\in H^1_{\mathrm{loc}}(\R^3;\Nb),
 \quad
 u_\infty(x)=\phi\left(\frac{x}{|x|}\right)
 \quad (x\in\R^3\backslash\{0\}).
\]
The final paragraph of the proof of Theorems~II and IV in \cite[p.~334]{SU82} shows, in three-dimensional source domain, that this blow-up is energy minimizing on compact subsets of $\R^3$ and has singular set exactly
$\{0\}$. Its restriction to the unit sphere,
\[
 \phi:=u_\infty|_{\mathbb S^2}:\mathbb S^2\to\Nb,
\]
is called the link and is smooth and nonconstant.  Zero homogeneity means
that $u_\infty$ is constant on each ray, so the link gives all its angular
dependence.  For $r>0$, the polar-coordinate tension equation
\[
 \tau(u_\infty)(r,\theta)
 =r^{-2}\tau_{g_{\mathrm r},g_F}(\phi)(\theta)
\]
shows that $\phi$ is harmonic.

Since $\mathbb{S}^2$ is simply connected, $\phi$ lifts through
$\PP:\mathbb{S}^3\to\Nb$:
\[
 \phi^\sharp:\mathbb{S}^2\to(\mathbb{S}^3,\PP^*g_F),
 \quad
 \PP\circ\phi^\sharp=\phi.
\]
Put $c:=8r_*^2$ and define
\[
 \widehat\PP:=\PP\circ\Theta^{-1},
 \quad
 \widetilde\phi:=\Theta\circ\phi^\sharp.
\]
By \eqref{eq:g4-pullback-isometry},
\[
 \widehat\PP^*g_F=cg_4,
 \quad
 \widehat\PP\circ\widetilde\phi=\phi.
\]
Thus, $\widehat\PP:(\mathbb{S}^3,cg_4)\to(\Nb,g_F)$ is a local isometry.
Constant rescaling does not change the Levi--Civita connection, so
\begin{equation}
 \tau_{g_{\mathrm r},g_4}(\widetilde\phi)=0.
 \label{eq:lifted-sphere-harmonic}
\end{equation}

For $x\neq0$, define
\[
 \widetilde u_\infty(x)
 :=\widetilde\phi\left(\frac{x}{|x|}\right).
\]
If
\[
 \widetilde W\in
 C_c^\infty(\R^3\backslash\{0\},\widetilde u_\infty^*T\mathbb{S}^3),
\]
its exponential variation projects under $\widehat\PP$ to an admissible
variation of $u_\infty$.  Local minimality, local isometry, and
\eqref{eq:homothety} give
\begin{equation}
 \IF_{\widetilde u_\infty}^{g_4}
 (\widetilde W,\widetilde W)\geq0.
 \label{eq:lifted-cone-stability}
\end{equation}

The next lemma transfers this cone stability to a shifted stability inequality
for its spherical link.

\begin{lem}\label{lem:shifted-link-stability}
For every smooth variation field $V$ along $\widetilde\phi$, with
$V(\theta)\in T_{\widetilde\phi(\theta)}\mathbb S^3$, one has
\begin{equation}
 \IF_{\widetilde\phi}^{g_4}(V,V)
 +\frac{1}{4}\int_{\mathbb{S}^2}|V|_{g_4}^2\,\ud A_{g_{\mathrm r}}\geq0.
 \label{eq:quarter-stability}
\end{equation}
\end{lem}

\begin{proof}
We separate the radial and angular parts of the cone index, then choose
radial functions approaching the sharp Hardy quotient.

\smallskip\noindent\textit{Separation of variables.}
For $ \widetilde W(r,\theta)=\eta(r)V(\theta) $ and $ \eta\in C_c^\infty(0,+\infty) $, the polar-coordinate expression for \eqref{eq:index-form-definition} is
\begin{align}
 \IF_{\widetilde u_\infty}^{g_4}(\widetilde W,\widetilde W)
 ={}&
 \left(\int_0^{+\infty}r^2|\eta'|^2\,\ud r\right)
 \int_{\mathbb{S}^2}|V|_{g_4}^2\,\ud A_{g_{\mathrm r}}+
 \left(\int_0^{+\infty}\eta^2\,\ud r\right)
 \IF_{\widetilde\phi}^{g_4}(V,V).
 \label{eq:cone-index-separation}
\end{align}
The left-hand side is nonnegative by \eqref{eq:lifted-cone-stability}.

\smallskip\noindent\textit{The radial quotient.}
Choose $0\neq\chi\in C_c^\infty(\R)$ and set
\[
 \eta_R(r):=r^{-1/2}\chi\left(\frac{\log r}{R}\right).
\]
The change of variables $s=(\log r)/R$, followed by the integration of the
cross term, gives
\[
 \frac{\int_0^{+\infty}r^2|\eta_R'|^2\,\ud r}{\int_0^{+\infty}\eta_R^2\,\ud r}
 =\frac{1}{4}+\frac{1}{R^2}
   \frac{\int_\R|\chi'|^2\,\ud s}{\int_\R\chi^2\,\ud s}
 \to\frac{1}{4}.
\]
Substitute $\eta_R$ into \eqref{eq:cone-index-separation}, divide by
$\int_0^{+\infty}\eta_R^2\,\ud r$, and let $R\to+\infty$.  This proves
\eqref{eq:quarter-stability}.  The shift $1/4$ is the radial Hardy cost
of extending a link variation to a three-dimensional cone.
\end{proof}

\section{Direct quantitative instability of harmonic two-spheres}
\label{sec:quantitative-instability}

The lifted link constructed in Section~\ref{sec:tangent-cones} is denoted by
$\widetilde\phi$.  In this section, we work with an arbitrary smooth
nonconstant harmonic map
\[
 \phi:(\mathbb S^2,g_{\mathrm r})\to(\mathbb S^3,g_4)
\]
and prove an estimate that will be applied to $\widetilde\phi$ in
Section~\ref{sec:cone-contradiction}.  We use $\phi$ here to keep the notation
for the general argument uncluttered.

The next proposition supplies the quantitative
instability estimate that will contradict the shifted stability inequality
\eqref{eq:quarter-stability}.

\begin{prop}[Quantitative instability]\label{prop:direct-quantitative-instability}
Let $\phi:(\mathbb S^2,g_{\mathrm r})\to(\mathbb S^3,g_4)$ be a
nonconstant smooth harmonic map.  There is a smooth section $ W\in C^\infty(\mathbb S^2,\phi^*T\mathbb S^3) $ such that
\[
 |W|_{g_4}^2\leq \frac{13}{4}
 \quad\text{on }\mathbb S^2
\]
and
\[
 \IF_\phi^{g_4}(W,W)\leq -8\pi.
\]
\end{prop}

The estimate is uniform over all such maps, including branched coverings.
Its pointwise bound is integrated with the round source measure in the
stability inequality, whereas its negative index will be computed using the
induced area measure.  We keep these two measures distinct throughout the proof.

\subsection{Branch points and smooth geometric bundles}
\label{subsec:branch-geometry}

The branched minimal structure of harmonic two-spheres is classical; see
Sacks--Uhlenbeck \cite[Corollary~1.7]{SU81} and the branched-immersion
theory of Gulliver--Osserman--Royden \cite{GOR73}.  We give the short
arguments needed here to describe the branch points and extend our fields
smoothly across them.

\begin{lem}\label{lem:conformal-sphere}
Define
\[
 \mathcal B:=\{p\in\mathbb S^2:\ud\phi|_p=0\}.
\]
Then $\phi$ is weakly conformal: $\phi^*g_4$ is a nonnegative scalar
multiple of $g_{\mathrm r}$ at each point.  Its restriction to
$\mathbb S^2\backslash\mathcal B$ is a conformal immersion with zero mean curvature, that is, a minimal
immersion.
\end{lem}

\begin{proof}
We use the quadratic differential of the harmonic map. In a holomorphic
coordinate $z=x+\mathrm i y$ with $g_{\mathrm r}=e^{2v}|\ud z|^2$, put
$\phi_z=\ud\phi((\partial_x-\mathrm i\partial_y)/2)$ and
$\partial_{\ol z}=(\partial_x+\mathrm i\partial_y)/2$.
Harmonicity and the torsion-free identity give
\[
 \nabla^{\phi^*g_4}_{\partial_{\ol z}}\phi_z
 =\frac{1}{4}e^{2v}\tau_{g_{\mathrm r},g_4}(\phi)=0.
\]
The quadratic differential
\[
 \mathcal H_\phi:=g_4(\phi_z,\phi_z)(\ud z)^2
\]
is independent of the holomorphic coordinate.  We extend $g_4$
complex-bilinearly; the squared norm of a complex vector $Z$ is
$g_4(Z,\overline Z)$, whereas $g_4(Z,Z)$ can vanish for $Z\ne0$.
Metric compatibility gives
\[
 \partial_{\ol z}g_4(\phi_z,\phi_z)
 =2g_4\left(\nabla^{\phi^*g_4}_{\partial_{\ol z}}\phi_z,\phi_z\right)=0.
\]
Thus $\mathcal H_\phi$ is holomorphic on $\mathbb S^2$. In a stereographic
coordinate its coefficient is entire and is $O(|z|^{-4})$ at infinity,
so Liouville's theorem makes it zero. Hence
$|\phi_x|_{g_4}^2=|\phi_y|_{g_4}^2$ and
$g_4(\phi_x,\phi_y)=0$, which is weak conformality. Off $\mathcal B$ these
identities give rank two. The induced metric is conformal to the source
metric, so conformal invariance of harmonicity makes the immersion harmonic
with respect to its induced metric. Its tension field is the trace of its
second fundamental form; therefore its mean curvature vanishes.
\end{proof}

\begin{rem}
By Lemma~\ref{lem:conformal-sphere}, at any regular point
$p\in\mathbb S^2\backslash\mathcal B$, conformality and
$\ud\phi|_p\neq0$ imply
\[
 (\phi^*g_4)_p=\lambda_p(g_{\mathrm r})_p
 \quad\text{for some }\lambda_p>0.
\]
Hence $\ud\phi|_p$ has real rank two and is an isomorphism from
$T_p\mathbb S^2$ to its image.
\end{rem}

The following local factorization makes the
vanishing order of $\ud\phi$ at a branch point explicit.

\begin{lem}\label{lem:local-branch-model}
For any $p\in\mathbb S^2$, there exist an open neighborhood
$U_p\subset\mathbb S^2$ and a holomorphic coordinate
$z=x+\mathrm i y:U_p\to \mathbb D\subset\mathbb C$ with $z(p)=0$ such that the following properties hold.
\begin{enumerate}
\item\label{assertion1} There is a real-analytic function $v$ on $U_p$, such that
\[
 g_{\mathrm r}=e^{2v}(\ud x^2+\ud y^2)=e^{2v}|\ud z|^2.
\]
\item\label{assertion2} There also exist an integer
$m_p\geq0$ and a smooth section along $\phi$ of the complexified pullback
tangent bundle, denoted by $ G\in C^\infty(
 U_p,\phi^*T\mathbb S^3\otimes_{\mathbb R}\mathbb C) $, such that
\[
 \phi_z:=\ud\phi\left(
 \frac{1}{2}(\partial_x-\mathrm i\partial_y)
 \right)=z^{m_p}G,
 \quad G(z)\ne0\quad\text{for any }z\in U_p.
\]
\item\label{assertion3} On $U_p\backslash\mathcal B$,
\[
 \phi^*g_4=|z|^{2m_p}\lambda|\ud z|^2
\]
with $\lambda:=2g_4(G,\overline G)$.  Here,
the bar denotes complex conjugation in the complexified pullback bundle,
and $g_4$ denotes the complex-bilinear extension of the real metric.  Thus, if
$G=G_1+\mathrm iG_2$ with
$G_1,G_2\in C^\infty(U_p,\phi^*T\mathbb S^3)$, then
\[
 g_4(G,\overline G)
 =|G_1|_{g_4}^2+|G_2|_{g_4}^2.
\]
In particular, $\lambda$ is smooth and positive on $U_p$.
\end{enumerate}
The integer $m_p$ is the vanishing order of $\phi_z$ at $p$; at a branch
point we call it the branch order. In particular, $m_p=0$ exactly at regular
points. The branch set $\mathcal B$ is finite.
\end{lem}

\begin{proof}
We change the local bundle frame by a smooth invertible matrix, called a
gauge, so that the new coefficients satisfy the ordinary Cauchy--Riemann
equation.  Their zeros will give the branch orders.

\smallskip\noindent\emph{Step 1: a local holomorphic gauge.}
Choose a holomorphic coordinate $z=x+\mathrm i y$ centered at $p$ in which
$g_{\mathrm r}=e^{2v}|\ud z|^2$ with $v$ real analytic. Put
$\partial_z=(\partial_x-\mathrm i\partial_y)/2$ and
$\partial_{\ol z}=(\partial_x+\mathrm i\partial_y)/2$.
Harmonicity and the torsion-free identity give
\[
 \nabla^{\phi^*g_4}_{\partial_{\ol z}}\phi_z
 =\frac{1}{4}e^{2v}\tau_{g_{\mathrm r},g_4}(\phi)=0.
\]
In a smooth complex frame $(E_1,E_2,E_3)$ of
$\phi^*T\mathbb S^3\otimes_{\mathbb R}\mathbb C$, write
$\phi_z=\sum_{\al=1}^3 u^\al E_\al$.  The smooth connection matrix
$\Gamma_{\ol z}$ is determined by
\[
 \nabla^{\phi^*g_4}_{\partial_{\ol z}}E_\beta
 =\sum_{\al=1}^3(\Gamma_{\ol z})^\al_\beta E_\al.
\]
For $\mathbf u=(u^1,u^2,u^3)^\T$, the equation becomes
$\partial_{\ol z}\mathbf u+\Gamma_{\ol z}\mathbf u=0$.

Take a closed coordinate disk $\overline{\mathbb D}_r$ contained in the
chart and centered at $0$.  Define
on continuous matrix-valued functions
\[
 (\mathcal C_0F)(z)
 :=\frac{1}{\pi}\int_{\mathbb D_r}
   \left(\frac{1}{z-\zeta}+\frac{1}{\zeta}\right)F(\zeta)\,\ud A(\zeta).
\]
Here $\ud A$ denotes Euclidean area in the coordinate disk.
This weakly singular integral is continuous on the closed disk: the integrals
over small disks about the singularities are uniformly small, and the
remaining kernel is continuous. The fundamental-solution identity
$\partial_{\ol z}(1/(\pi z))=\delta_0$, where $\delta_0$ is the Dirac
mass at $0$, gives
$\partial_{\ol z}\mathcal C_0F=F$ distributionally, while
$(\mathcal C_0F)(0)=0$. Moreover, for $|z|\leq r$,
$\mathbb D_r\subset B_{2r}(z)$, so
\[
 \frac{1}{\pi}\int_{\mathbb D_r}
       \left(\frac{1}{|z-\zeta|}+\frac{1}{|\zeta|}\right)\ud A(\zeta)
 \leq 6r,
 \quad
 \|\mathcal C_0F\|_{C^0}\leq6r\|F\|_{C^0}.
\]
All $C^0$ norms here are supremum norms on $\overline{\mathbb D}_r$,
using the matrix operator norm. Decrease $r$ so that
$\rho:=6r\|\Gamma_{\ol z}\|_{C^0}<1/2$. The map
$P\mapsto\mathbf I_3-\mathcal C_0(\Gamma_{\ol z}P)$ is then a contraction
on $C^0(\overline{\mathbb D}_r,M_3(\mathbb C))$. Its fixed point satisfies
\[
 \partial_{\ol z}P=-\Gamma_{\ol z}P,
 \quad P(0)=\mathbf I_3,
 \quad \|P-\mathbf I_3\|_{C^0}\leq\frac{\rho}{1-\rho}<1.
\]
Thus $P$ is invertible. Interior $L^p$ estimates for $\partial_{\ol z}$,
followed by bootstrapping with the smooth coefficient $\Gamma_{\ol z}$,
give $P\in C^\infty(\mathbb D_r)$. Consequently,
\[
 \partial_{\ol z}(P^{-1}\mathbf u)
 =P^{-1}\Gamma_{\ol z}\mathbf u-P^{-1}\Gamma_{\ol z}\mathbf u=0.
\]

\smallskip\noindent\emph{Step 2: nonconstancy and finite vanishing order.}
The preceding gauge construction applies at every point and uses only
harmonicity. Suppose that $\phi=q_0$ on a nonempty open set, and let
\[
 \mathcal O:=\{x\in\mathbb S^2:\phi=q_0
                         \text{ in a neighborhood of }x\}.
\]
This set is nonempty and open. For $x\in\overline{\mathcal O}$, choose a
connected coordinate disk about $x$ with the matrix $P$ constructed above. It meets $\mathcal O$ in a nonempty open
set, on which $\phi_z=0$. The holomorphic vector $P^{-1}\mathbf u$ then
vanishes on the whole disk by the identity theorem. Thus $\ud\phi=0$ on
the disk, and its constant value is $q_0$ because it meets $\mathcal O$.
Hence $x\in\mathcal O$, so $\mathcal O$ is closed. Connectedness of
$\mathbb S^2$ would make $\phi$ constant, a contradiction.

Consequently, the holomorphic vector $P^{-1}\mathbf u$ is not identically
zero on any such disk.  Let $m_p$ be the largest integer for which every
component is divisible by $z^{m_p}$.  Factoring out this power gives
$P^{-1}\mathbf u=z^{m_p}\mathbf h$ for an integer $m_p\geq0$ and a
holomorphic vector $\mathbf h$ with $\mathbf h(0)\ne0$. Therefore,
\[
 \phi_z=z^{m_p}G,
 \quad G:=\sum_\al(P\mathbf h)^\al E_\al.
\]
After shrinking the coordinate neighborhood, $G$ is smooth and nowhere
zero. Thus $m_p$ is the vanishing order of $\phi_z$, and $m_p=0$ exactly
when $\ud\phi|_p\ne0$. Moreover, a sufficiently small neighborhood contains
no branch point other than $p$. The branch set is therefore discrete;
since it is closed by smoothness of $\phi$ and $\mathbb S^2$ is compact,
it is finite.

\smallskip\noindent\emph{Step 3: the induced metric.}
The conformality from Lemma~\ref{lem:conformal-sphere} gives $g_4(G,G)=0$, first off $z=0$ by cancellation and then
at $0$ by continuity. Since $G\ne0$, the smooth function
$\lambda:=2g_4(G,\overline G)$ is positive. Thus, on the regular set,
\[
 \phi^*g_4=e^{2u}|\ud z|^2,
 \quad e^{2u}=|z|^{2m_p}\lambda,
 \quad \lambda>0,
\]
which proves the stated metric factorization and completes the proof.
\end{proof}

At a branch point, the differential has zero image.  The ramified tangent
plane means the limit of the nearby oriented tangent planes; it remains
two-dimensional and supplies a unit normal.  The local factorization gives
the smooth extension in the next lemma.

\begin{lem}\label{lem:ramified-tangent-normal-bundles}
For any $q\in\mathbb S^2\backslash\mathcal B$, define the tangent plane of
$\phi$ at $q$ by
\[
 (\mathcal T_\phi)_q
 :=\ud\phi_q(T_q\mathbb S^2)
 \subset T_{\phi(q)}\mathbb S^3,
\]
oriented so that $\ud\phi_q:T_q\mathbb S^2\to(\mathcal T_\phi)_q$ preserves
orientation. Then the following properties hold.
\begin{enumerate}
\item These oriented planes extend smoothly through the branch points
to a rank-two sub-bundle $\mathcal T_\phi\subset\phi^*T\mathbb S^3$, called the ramified tangent bundle.
\item The oriented unit normals of $ \mathcal{T}_{\phi} $ extend to a smooth vector field along $\phi$,
\[
 N\in C^\infty(\mathbb S^2,\phi^*T\mathbb S^3),
 \quad
 N(q)\in T_{\phi(q)}\mathbb S^3,
 \quad
 |N(q)|_{g_4}=1,
 \quad q\in\mathbb S^2,
\]
such that
\[
 (\mathcal T_\phi)_q
 =\{V\in T_{\phi(q)}\mathbb S^3:g_4(V,N(q))=0\},
 \quad q\in\mathbb S^2.
\]
\item The orthogonal complement $ \nu_\phi:=\mathcal T_\phi^{\perp_{g_4}} $ is a smooth normal line bundle that satisfies
\[
 (\nu_\phi)_q
 =\operatorname{span}_{\mathbb R}\{N(q)\}
 \subset T_{\phi(q)}\mathbb S^3,
 \quad q\in\mathbb S^2.
\]
\end{enumerate}
\end{lem}

\begin{proof}
The local factorization supplies a smooth oriented frame for the limiting
plane; we use it to construct a normal and then glue the local extensions.

\smallskip\noindent\textit{The limiting tangent plane.}
In a branch chart, write $\phi_z=z^{m_p}G$ as in
Lemma~\ref{lem:local-branch-model} and set
\begin{equation}
 \mathcal E_1:=\operatorname{Re}G,
 \quad \mathcal E_2:=-\operatorname{Im}G,
 \quad \mathcal T(z):=\operatorname{span}_{\mathbb R}
                         \{\mathcal E_1(z),\mathcal E_2(z)\}.
 \label{eq:extended-tangent-plane-local}
\end{equation}
The identities $g_4(G,G)=0$ and $g_4(G,\overline G)>0$ give
\[
 g_4(\mathcal E_1,\mathcal E_2)=0,
 \quad
 |\mathcal E_1|_{g_4}^2=|\mathcal E_2|_{g_4}^2
             =\frac{1}{2}g_4(G,\overline G)>0.
\]
Thus $\mathcal T$ is a smooth oriented plane field, including at $z=0$.
For $z\ne0$,
\[
 \left(\frac{1}{2}\phi_x,\frac{1}{2}\phi_y\right)
 =\left(\operatorname{Re}(z^{m_p}G),
               -\operatorname{Im}(z^{m_p}G)\right).
\]
Multiplication by $z^{m_p}$ has positive real determinant $|z|^{2m_p}$.
Hence $\mathcal T(z)$ is the tangent plane with the prescribed orientation,
and $\mathcal T(0)$ is its limit.

\smallskip\noindent\textit{The global normal.}
Using the metric cross product from Lemma~\ref{lem:berger-hopf-field},
extended complex-bilinearly, define
\[
 N(z):=
 \frac{-\mathrm iG(z)\mathbin{\times_{g_4}}\overline{G(z)}}{|-\mathrm iG(z)\mathbin{\times_{g_4}}\overline{G(z)}|_{g_4}}.
\]
Its numerator equals
$2\mathcal E_1\mathbin{\times_{g_4}}\mathcal E_2\ne0$, so $N$ is a
smooth real unit normal across the branch point. Local extensions agree
on the regular part of each overlap, because their orientations agree
there; continuity and density give agreement everywhere. They therefore
define a global smooth unit normal along $\phi$.

Now set $\mathcal T_\phi:=N^{\perp_{g_4}}$ and
$\nu_\phi:=\operatorname{span}_{\mathbb R}\{N\}$. These smooth bundles
coincide with the tangent and normal bundles on the regular set and with
the local construction above at every branch point. In particular,
$\mathcal T_\phi$ has rank two where $\ud\phi=0$, and
$\phi^*T\mathbb S^3=\mathcal T_\phi\oplus\nu_\phi$ is the required
orthogonal splitting.
\end{proof}

\subsection{Surface identities from the Hopf symmetry}
\label{subsec:surface-identities}

The smooth normal now defines a global angle with the Hopf field.
We derive the identities that control
the second fundamental form and the angle between the normal and the Hopf
field.

\begin{lem}\label{lem:berger-surface-identities}
Let $\xi=V_{\mathrm H}/2$ be given by \eqref{eq:g4-unit-Hopf-field}, and
let $N$ be the global unit normal supplied by
Lemma \ref{lem:ramified-tangent-normal-bundles}. Set
\begin{equation}
 C:=g_4(N,\xi\circ\phi)\in C^\infty(\mathbb S^2),
 \quad
 \gamma:=\phi^*(\xi^\flat)
 \in C^\infty(\mathbb S^2,T^*\mathbb S^2).                    \label{defCgamma}
\end{equation}
The scalar $C$ is the cosine of the angle between the unit normal and the
unit Hopf field, so $|C|\leq1$.  The one-form $\gamma$ measures the Hopf component of a
tangent direction: $\gamma(X)=g_4(\ud\phi(X),\xi\circ\phi)$.
Let $B$ be the scalar second fundamental form on
$\mathbb S^2\backslash\mathcal B$.  It measures the normal component of
the derivative of a tangent direction.  Then the following properties hold.
\begin{enumerate}
\item $ B\in C^\infty\left(
 \mathbb S^2\backslash\mathcal B,
 T^*(\mathbb S^2\backslash\mathcal B)
 \otimes T^*(\mathbb S^2\backslash\mathcal B)\right) $ and for the tangent vector fields $X,Z$ on the regular set, it is defined by
\begin{equation}
 B(X,Z):=g_4\left(
\nabla_X^{\phi^*g_4}\bigl(\ud\phi(Z)\bigr),N
 \right).                                                       \label{defB}
\end{equation}
\item On source one-forms, let $*$ be the oriented quarter-turn, given in
an oriented conformal chart by $*\ud x=\ud y$ and $*\ud y=-\ud x$.
This is the Hodge star.  On $\mathbb S^2\backslash\mathcal B$,
\[
 B=-\frac{3}{2}\bigl((*\gamma)\otimes\gamma
                       +\gamma\otimes(*\gamma)\bigr),
 \quad
 \ud C=-\frac{1+3C^2}{2}(*\gamma).
\]
\end{enumerate}
\end{lem}

\begin{proof}
The Abresch--Rosenberg differential combines the second fundamental form
and the Hopf field into a holomorphic quadratic differential on the regular
surface.  We extend it through the branch points; its vanishing on the
sphere then determines $B$ in terms of $\gamma$.  On the regular set,
write $g:=\phi^*g_4$ for the induced metric.

\smallskip\noindent\textit{Branch-point bounds.}
Choose $p\in\mathbb S^2$ and a branch chart $z=x+\mathrm i y$ provided
by Lemma \ref{lem:local-branch-model}. On the regular part of this
chart, write
\[
 \gamma=a\,\ud x+b\,\ud y
\]
and define
\[
 B_{ij}:=
 g_4\left(
  \nabla^{\phi^*g_4}_{\partial_i}\bigl(\ud\phi(\partial_j)\bigr),N
 \right).
\]
Thus, $B_{ij}=B(\partial_i,\partial_j)$ for the tensor in \eqref{defB}.

Since $\gamma=\phi^*(\xi^\flat)$, we have
$a=g_4(\phi_x,\xi\circ\phi)$ and
$b=g_4(\phi_y,\xi\circ\phi)$.  Lemma
\ref{lem:local-branch-model} gives
$\ud\phi=O(|z|^{m_p})$, and hence
$a,b=O(|z|^{m_p})$.  It also gives
\[
 \phi_x=z^{m_p}G+\overline z^{m_p}\overline G,
 \quad
 \phi_y=\mathrm i(z^{m_p}G-\overline z^{m_p}\overline G).
\]
For $m_p=0$, smoothness gives $B_{ij}=O(1)$.  If
$m_p\geq1$, the terms of order $m_p-1$ obtained by differentiating the
explicit powers are linear combinations of $G$ and $\overline G$.  By
\eqref{eq:extended-tangent-plane-local}, they are tangent and therefore orthogonal
to $N$.  Every other term contributing to the normal component has a
factor $z^{m_p}$ or $\overline z^{m_p}$. Thus $B_{ij}=O(|z|^{m_p})$ in both cases.  Combining this estimate with those for
$a$ and $b$ gives
\begin{equation}
 a,b=O(|z|^{m_p}),
 \quad
 B_{ij}=O(|z|^{m_p}).\label{abBijestimate}
\end{equation}

\smallskip\noindent\textit{Vanishing of the quadratic differential.}
In the notation of \cite[Section~2.1]{TU12}, take
$(\kappa,\tau)=(4,2)$.  Then
\[
 g_{4,2}
 =g_{\mathrm r}+3g_{\mathrm r}(\,\cdot\,,V_{\mathrm H})^2
 =g_4.
\]
In this identification, the unit vertical field is
$\xi=V_{\mathrm H}/2$, and \eqref{eq:bundle-curvature-two} reads
$\nabla_X^{g_4}\xi=2X\mathbin{\times_{g_4}}\xi$. Hence, the parameter denoted
by $\tau$ in \cite{TU12} equals $2$.  Thus, our
oriented Berger sphere is precisely $(\mathbb S^3,g_{4,2})$ in \cite{TU12}.
Since $B$ is symmetric and $\phi$ is minimal by
Lemma~\ref{lem:conformal-sphere},
$B_{22}=-B_{11}$ in $ \mathbb{S}^2\backslash\mathcal{B}$.  Define on the regular set
\[
\begin{aligned}
 \mathfrak a
 &:=g_4(\phi_z,\xi\circ\phi)=\frac{1}{2}(a-\mathrm i b),\\
 p
 &:=g_4\left(
    \nabla^{\phi^*g_4}_{\partial_z}\phi_z,N
  \right)
  =\frac{1}{4}(B_{11}-B_{22}-2\mathrm iB_{12})
  =\frac{1}{2}(B_{11}-\mathrm iB_{12}).
\end{aligned}
\]
For $H=0$ (where the convention in \cite{TU12} is
$H=(\operatorname{tr}_g B)/2$) and
$(\kappa,\tau)=(4,2)$, the
Abresch--Rosenberg differential in our notation \cite[(3.2)]{TU12} is
\[
 \mathcal Q_\phi=4\mathrm i\bigl(p-3\mathrm i\mathfrak a^2\bigr)(\ud z)^2.
\]
It is holomorphic on the regular set
because the Abresch--Rosenberg differential
\cite[(3.2)]{TU12} is holomorphic for any constant-mean-curvature
immersion, and here $H\equiv0$. By the estimates \eqref{abBijestimate}, the
displayed definitions of $\mathfrak a$ and $p$ give
$\mathfrak a,p=O(|z|^{m_p})$. Hence, the coefficient
$4\mathrm i(p-3\mathrm i\mathfrak a^2)$ of $\mathcal Q_\phi$ is bounded at any
branch point.  The removable singularity theorem extends
$\mathcal Q_\phi$ holomorphically to $\mathbb S^2$.  The stereographic-coordinate argument in the proof of
Lemma~\ref{lem:conformal-sphere} shows that a holomorphic quadratic
differential on $\mathbb S^2$ vanishes.  Hence $\mathcal Q_\phi=0$. Therefore, $p=3\mathrm i\mathfrak a^2$, which, together with $B_{22}=-B_{11}$, is
equivalent to
\begin{equation}
 B_{11}=3ab,\quad
 B_{12}=-\frac{3}{2}(a^2-b^2),\quad
 B_{22}=-3ab.                                      \label{eq:direct-step-1}
\end{equation}
\smallskip\noindent\textit{Real surface identities.}
For $H=0$ and $(\kappa,\tau)=(4,2)$, the last three compatibility equations
in \cite[Proposition~3.1, (3.3)]{TU12} read
\[
 C_z=2\mathrm i\mathfrak a
       -2e^{-2u}\overline{\mathfrak a}\,p,
 \quad
 \mathfrak a_{\ol z}=\mathrm i e^{2u}C,
 \quad
 |\mathfrak a|^2=\frac{e^{2u}}{4}(1-C^2).
\]
Using $p=3\mathrm i\mathfrak a^2$ and
$|\mathfrak a|^2=(1/4)e^{2u}(1-C^2)$ in the formula displayed for $C_z$
gives
\begin{equation}
 C_z=\frac{\mathrm i}{2}(1+3C^2)\mathfrak a,
 \quad
 \mathfrak a_{\ol z}=\mathrm i e^{2u}C,
 \quad
 |\mathfrak a|^2=\frac{e^{2u}}{4}(1-C^2).
\label{eq:berger-complex-identities}
\end{equation}
Since $\mathfrak a=(a-\mathrm i b)/2$, the equations
in \eqref{eq:berger-complex-identities} are equivalent to
\begin{equation}
\begin{aligned}
 a_x+b_y&=0,
 &a_y-b_x&=4e^{2u}C,\\
 C_x&=\frac{1}{2}(1+3C^2)b,
 &C_y&=-\frac{1}{2}(1+3C^2)a,\\
 a^2+b^2&=e^{2u}(1-C^2).
\end{aligned}                                      \label{eq:direct-step-2}
\end{equation}
These are precisely the zero-mean-curvature, zero
Abresch--Rosenberg-differential specialization of the compatibility
equations for surfaces in the Berger sphere $g_4$.

Since $g=\phi^*g_4$ is conformal to $g_{\mathrm r}$ on the regular set and
the Hodge star on one-forms is conformally invariant in dimension two, the operator $*$ in the statement is the restriction of the round-sphere
Hodge star.  In
particular, $*$ itself does not have a singularity at a branch point.
Equations \eqref{eq:direct-step-1}--\eqref{eq:direct-step-2} have coordinate-free forms
\begin{equation}
 B=-\frac{3}{2}\bigl((*\gamma)\otimes\gamma
                       +\gamma\otimes(*\gamma)\bigr),               \label{eq:direct-step-3}
\end{equation}
and, on the regular set,
\begin{equation}
 \ud C=-\frac{1+3C^2}{2}(*\gamma), \label{eq:direct-step-4}
\end{equation}
where $C$ is the angle cosine and $\gamma$ is the Hopf one-form from
\eqref{defCgamma}.
\end{proof}

\subsection{The Hopf-adapted test field}
\label{subsec:canonical-variation}

The tangential component of the Hopf field gives a test field with a
coefficient depending only on the angle $C$.  Its formula remains smooth
where $\ud\phi$ vanishes.

\begin{lem}[The Hopf-adapted variation]\label{lem:canonical-field}
Let $\xi=V_{\mathrm H}/2$ be the unit Hopf Killing field.  Let $N$ be
the smooth unit normal along $\phi$ from
Lemma~\ref{lem:ramified-tangent-normal-bundles}, and set
$C=g_4(N,\xi\circ\phi)$.  Define
\begin{equation}
 T:=\xi\circ\phi-CN,\quad
 k(C):=\frac{6C}{1+3C^2},\quad
 W:=N+k(C)T.
 \label{eq:canonical-field}
\end{equation}
Then $T$ is a smooth section of $\mathcal T_\phi$, and $W$ is a smooth
section of $\phi^*T\mathbb S^3$ satisfying $|W|_{g_4}^2\leq13/4$ everywhere.
\end{lem}
\begin{proof}
We use the orthogonal decomposition of the unit Hopf field to estimate $W$.
The fields $N$ and $\xi\circ\phi$ are smooth on the whole sphere, so
$T$ is smooth, orthogonal to $N$, and
\begin{equation}
 |T|_{g_4}^2=1-C^2.
 \label{eq:tangent-Hopf-length}
\end{equation}
Since $1+3C^2>0$, the field $W$ is smooth.  Put $t:=C^2\in[0,1]$.
Orthogonality gives
\[
 |W|_{g_4}^2
 =1+\frac{36t(1-t)}{(1+3t)^2}\leq\frac{13}{4},
\]
because $(1+3t)^2-16t(1-t)=(1-5t)^2\geq0$.  The construction uses only
smooth fields along $\phi$, so it is valid at branch points as well.
\end{proof}

\subsection{The Jacobi potential and its total mass}
\label{subsec:jacobi-mass}

The index calculation uses the second fundamental form and ambient Ricci
curvature through a single scalar potential.  The surface identities also
determine a conformal metric whose area measures its total mass.

\begin{lem}\label{lem:jacobi-potential-formula}
On $ \mathbb{S}^2\backslash\mathcal{B} $, let
$g:=\phi^*g_4$ and
\[
 q:=|B|_g^2+\operatorname{Ric}_{g_4}(N,N).
\]
We use $\Delta_g=\operatorname{div}_g\nabla^g$.
For normal area variations $fN$ with smooth scalar amplitude $f$, the
Jacobi operator acts by
$f\mapsto\Delta_g f+qf$ \cite[(3.4)]{TU12}; this explains the name
Jacobi potential.  The two terms in $q$ come from surface bending and
ambient Ricci curvature.  We have
\begin{equation}
 q=\frac{1}{2}(1+3C^2)^2.\label{eq:direct-step-16}
\end{equation}
Here $N$ is the smooth unit normal along $\phi$, $\xi$ is the unit Hopf
Killing field, and $C=g_4(N,\xi\circ\phi)$ is their angle function
from \eqref{defCgamma}.
\end{lem}

\begin{proof}
We compute the two terms in $q$ and then derive the angle equations.  All computations use the induced metric
$g=\phi^*g_4$ on the regular set.

\smallskip\noindent\textit{The Jacobi potential.}
The one-form $\gamma$ is dual to the tangential Hopf field, so
$|\gamma|_g^2=1-C^2$ by \eqref{eq:tangent-Hopf-length}.
The two one-forms $\gamma$ and $*\gamma$ are orthogonal and have the same
length.  Therefore the formula for $B$ in
\eqref{eq:direct-step-3} gives
\begin{equation}
 |B|_g^2=\frac{9}{2}(1-C^2)^2.
 \label{eq:shape-squared}
\end{equation}
For $(\kappa,\tau)=(4,2)$, the ambient Ricci term in
\cite[(3.4)]{TU12} is
\begin{equation}
 \operatorname{Ric}_{g_4}(N,N)
 =2\tau^2+(\kappa-4\tau^2)(1-C^2)
 =-4+12C^2.
 \label{eq:normal-ricci}
\end{equation}
Adding these terms yields
\[
 q=\frac{9}{2}(1-C^2)^2-4+12C^2
  =\frac{1}{2}(1+3C^2)^2,
\]
which proves \eqref{eq:direct-step-16}.
In particular, the negative Ricci curvature in some directions does not
prevent the full Jacobi potential from being positive.

\smallskip\noindent\textit{The angle equations.}
Put $s:=1+3C^2$.  In a conformal chart with
$g=e^{2u}|\ud z|^2$ and $\gamma=a\,\ud x+b\,\ud y$,
\eqref{eq:direct-step-2} gives
\[
 C_x=\frac{s}{2}b,\quad C_y=-\frac{s}{2}a,\quad
 b_x-a_y=-4e^{2u}C,\quad a^2+b^2=e^{2u}(1-C^2).
\]
Since $s_x=3Csb$ and $s_y=-3Csa$, differentiation gives
\[
 \begin{aligned}
 (\partial_x^2+\partial_y^2)C
 &=\frac{1}{2}\bigl(s_xb-s_ya+s(b_x-a_y)\bigr)\\
 &=\frac{s}{2}\bigl(3C(a^2+b^2)-4e^{2u}C\bigr)
 =-\frac{1}{2}e^{2u}Cs^2.
 \end{aligned}
\]
Thus,
\begin{equation}
 \Delta_gC=-\frac{1}{2}C(1+3C^2)^2,
 \quad (\Delta_g+q)C=0.
 \label{eq:direct-step-14}
\end{equation}
The second equality is also the Jacobi equation for the normal component of
an ambient Killing field.  Finally, \eqref{eq:direct-step-4} gives
\begin{equation}
 |\nabla C|_g^2
 =\frac{1}{4}(1+3C^2)^2|\gamma|_g^2
 =\frac{1}{4}(1+3C^2)^2(1-C^2).
 \label{eq:angle-gradient-C}
\end{equation}
These equations establish the potential formula and the differential
identities needed for the conformal metric below.
\end{proof}

The next lemma rescales the induced metric so that the Jacobi potential becomes
a spherical area density.

\begin{lem}\label{lem:spherical-jacobi-metric}
Define on $\mathbb S^2\backslash\mathcal B$ the
metric
\[
 g_{\circ}:=\frac{q}{2}g.
\]
Then $g_{\circ}$ has Gaussian curvature $1$ on
the regular set.  Moreover, in any branch chart from Lemma
\ref{lem:local-branch-model}, it has the form
\[
 g_{\circ}=|z|^{2m_p}\lambda_{\circ}|\ud z|^2,
 \quad \lambda_{\circ}\in C^\infty(U_p),\quad
 \lambda_{\circ}>0.
\]
Thus, $g_{\circ}$ has the same branch order $m_p$ as $\phi$ at any
$p\in\mathcal B$.  Consequently,
\begin{equation}
 \int_{\mathbb S^2\backslash\mathcal B}q\,\ud A_g
 =8\pi+4\pi\sum_{p\in\mathcal B}m_p
 \geq8\pi.
 \label{eq:spherical-jacobi-potential}
\end{equation}
\end{lem}

\begin{proof}
We compute the curvature under the conformal rescaling and then apply
Gauss--Bonnet with the branch orders included.

\smallskip\noindent\emph{Step 1: curvature on the regular set.}
Put $s:=1+3C^2$ and $t:=C^2$.  Since $q=s^2/2$, the new metric is
$g_{\circ}=(s^2/4)g$.  The Gauss equation
\cite[(3.1)]{TU12}, specialized to $H=0$ and $(\kappa,\tau)=(4,2)$,
and the formula $|B|_g^2=(9/2)(1-C^2)^2$ give
\[
 K_g=-\frac{1}{2}|B|_g^2+4-12C^2
     =\frac{7-30t-9t^2}{4}.
\]
The angle equations \eqref{eq:direct-step-14} and
\eqref{eq:angle-gradient-C} yield
\[
\begin{aligned}
 \Delta_g\log s
 &=\frac{6C}{s}\Delta_gC
   +\frac{6(1-3C^2)}{s^2}|\nabla C|_g^2\\
 &=-3ts+\frac{3}{2}(1-3t)(1-t)
 =\frac{3-18t-9t^2}{2}.
\end{aligned}
\]
Therefore
\[
 K_g-\Delta_g\log(s/2)
 =\frac{1+6t+9t^2}{4}
 =\frac{s^2}{4}.
\]
The conformal curvature formula now gives
\[
 K_{g_{\circ}}
 =\frac{4}{s^2}\bigl(K_g-\Delta_g\log(s/2)\bigr)=1
 \quad\text{on }\mathbb S^2\backslash\mathcal B.
\]

\smallskip\noindent\emph{Step 2: branch orders and Gauss--Bonnet.}
In a branch chart from Lemma~\ref{lem:local-branch-model},
\[
 g_{\circ}=|z|^{2m_p}\lambda_{\circ}|\ud z|^2,
 \quad \lambda_{\circ}:=\frac{s^2\lambda}{4}>0.
\]
The coefficient $\lambda_{\circ}$ is smooth, so the rescaling preserves
the branch order $m_p$.  Choose pairwise disjoint coordinate disks
$\mathbb D_{\varepsilon,p}:=\{|z_p|<\varepsilon\}$ and set
\[
 M_\varepsilon
 :=\mathbb S^2\backslash
   \bigcup_{p\in\mathcal B}\mathbb D_{\varepsilon,p}.
\]
Its Euler characteristic is $2-|\mathcal B|$, and Gauss--Bonnet gives
\begin{equation}
 \int_{M_\varepsilon}K_{g_{\circ}}\,\ud A_{g_{\circ}}
 +\sum_{p\in\mathcal B}
   \int_{\partial\mathbb D_{\varepsilon,p}}
      k_{g_{\circ}}\,\ud s_{g_{\circ}}
 =2\pi(2-|\mathcal B|).
 \label{eq:gauss-bonnet-punctured-sphere}
\end{equation}
Each inner boundary has the clockwise orientation.  Writing
$\lambda_{\circ}=e^{2\omega}$ and
$u=m_p\log r+\omega$, its Euclidean left normal is $n_0=\partial_r$
and its Euclidean geodesic curvature is $k_0=-1/r$.  The conformal
geodesic-curvature formula therefore gives
\[
\begin{aligned}
 \int_{\partial\mathbb D_{\varepsilon,p}}
       k_{g_{\circ}}\,\ud s_{g_{\circ}}
 &=\int_{\partial\mathbb D_{\varepsilon,p}}
       (k_0-\partial_{n_0}u)\,\ud s_0\\
 &=\int_0^{2\pi}
       \left(-\frac{1+m_p}{\varepsilon}+O(1)\right)
       \varepsilon\,\ud\theta\\
 &=-2\pi(m_p+1)+O(\varepsilon).
\end{aligned}
\]
The area density $|z|^{2m_p}\lambda_{\circ}$ is locally integrable.
Letting $\varepsilon\to0$ in
\eqref{eq:gauss-bonnet-punctured-sphere} thus yields
\[
 \int_{\mathbb S^2\backslash\mathcal B}
       K_{g_{\circ}}\,\ud A_{g_{\circ}}
 =4\pi+2\pi\sum_{p\in\mathcal B}m_p.
\]
Finally, $K_{g_{\circ}}=1$ and
$\ud A_{g_{\circ}}=(q/2)\ud A_g$, so
\[
 \int_{\mathbb S^2\backslash\mathcal B}q\,\ud A_g
 =8\pi+4\pi\sum_{p\in\mathcal B}m_p\geq8\pi.
\]
Each branch order contributes $4\pi$ to the total Jacobi mass, so
branching strengthens the lower bound in
\eqref{eq:spherical-jacobi-potential}.
\end{proof}

\subsection{Direct evaluation of the energy index}
\label{subsec:direct-index}

The field $K=\xi\circ\phi$ along the map comes from target isometries.
We will prove that it is an energy Jacobi field, meaning that
$\IF_\phi^{g_4}(V,K)=0$ for every smooth field $V$ along $\phi$.
Pairing with $K$ then reduces the index of $W$ to the integral of $q$ and
a Laplacian with zero integral.

\begin{lem}\label{lem:canonical-index-identity}
For the field $W$ in Lemma~\ref{lem:canonical-field},
\[
 \IF_\phi^{g_4}(W,W)
 =-\int_{\mathbb S^2\backslash\mathcal B}q\,\ud A_g.
\]
Here $g=\phi^*g_4$ is the induced metric on the regular set and
$q=|B|_g^2+\operatorname{Ric}_{g_4}(N,N)$ is the potential from
Lemma~\ref{lem:jacobi-potential-formula}.  The energy index on the left
uses the round source metric $g_{\mathrm r}$.
\end{lem}

\begin{proof}
We use the Jacobi field from the Hopf isometries to compute the
energy index directly.  We use the symmetric bilinear form
\eqref{eq:bilinear-index}, obtained from the quadratic second variation
\eqref{eq:index-form-definition}.  Recall that $N$ is
the smooth unit normal along $\phi$ and $\xi$ is the unit Hopf Killing
field.  We write $K=\xi\circ\phi$ for this field along the map and set
\[
 s:=1+3C^2,\quad
 a:=\frac{1-3C^2}{s},\quad b:=\frac{6C}{s}.
\]
These fields and coefficients are globally smooth.  Since $|K|=1$ and
$C=\langle N,K\rangle$, the defining formula for the canonical field gives
$W=N+b(K-CN)=aN+bK$.

\smallskip\noindent\textit{Pairing with the Killing field.}
For a smooth map $F:\mathbb S^2\to\mathbb S^3$, write
\[
 \mathcal E(F):=\frac{1}{2}\int_{\mathbb S^2}
          |\ud F|_{g_{\mathrm r},g_4}^2\,\ud A_{g_{\mathrm r}}
\]
for the Dirichlet energy on the round source, and let $\Psi_t$ be the
flow of the Killing field $\xi$.  For a smooth variation $\phi_u$ with $\phi_0=\phi$, set
$V:=\partial_u\phi_u|_{u=0}$.  Target isometry invariance gives
$\mathcal E(\Psi_t\circ\phi_u)=\mathcal E(\phi_u)$.  Harmonicity of
$\phi$ makes its first variation zero.  The mixed derivative at
$(t,u)=(0,0)$ is therefore its energy Hessian, which gives
\[
 \IF_\phi^{g_4}(V,K)=0
 \quad\text{for every smooth field }V\text{ along }\phi.
\]
Thus $K$ is a global Jacobi field.  In particular this identity applies to
$V=abN$ and $V=b^2K$.

Define the smooth one-forms
$\alpha(X):=\langle\nabla_XN,K\rangle$ and
$\beta(X):=\langle N,\nabla_XK\rangle$.  These derivatives of $N$ and
$K$ use the pullback connection $\nabla^{\phi^*g_4}$.
On the regular set let $g:=\phi^*g_4$.  In a
$g$-orthonormal frame $(e_1,e_2)$, the Weingarten formula gives
\[
 \alpha(e_i)=-\sum_{j=1}^2B(e_i,e_j)\gamma(e_j).
\]
The tensor formula \eqref{eq:direct-step-3}, the angle identity
\eqref{eq:direct-step-4}, and $|\gamma|_g^2=1-C^2$ give
\[
 \alpha=\frac{3}{2}(1-C^2)*\gamma
       =-\frac{3(1-C^2)}s\,\ud C,
 \quad
 \beta=\ud C-\alpha=\frac{4}{s}\,\ud C.
\]
Here $\ud C=\alpha+\beta$ follows by differentiating
$C=\langle N,K\rangle$.

Apply the product rule in the index formula
\eqref{eq:index-form-definition} and its bilinear form
\eqref{eq:bilinear-index} to the fields $aN$ and $bK$.
Unsubscripted brackets pair target vectors with $g_4$; the pairing
$\langle\alpha,\beta\rangle_g$ of source one-forms uses $g^{-1}$.
Subtracting
$\IF_\phi^{g_4}(b^2K,K)=0$ and
$\IF_\phi^{g_4}(abN,K)=0$ gives the first two identities below; unit length
of $N$ gives the third.  By conformal invariance, their regular-set
expressions are
\[
\begin{aligned}
 \IF_\phi^{g_4}(bK,bK)
 &=\int_{\mathbb S^2\backslash\mathcal B}|\ud b|_g^2\,\ud A_g,\\
 \IF_\phi^{g_4}(aN,bK)
 &=\int_{\mathbb S^2\backslash\mathcal B}
   \bigl(C\langle\ud a,\ud b\rangle_g
         +a\langle\alpha-\beta,\ud b\rangle_g\bigr)\,\ud A_g,\\
 \IF_\phi^{g_4}(aN,aN)
 &=\int_{\mathbb S^2\backslash\mathcal B}
   \bigl(|\ud a|_g^2
      +a^2(|B|_g^2-\operatorname{Ric}_{g_4}(N,N))\bigr)\,\ud A_g.
\end{aligned}
\]
In the last line the Weingarten formula supplies $|B|_g^2$, while the
curvature trace over the two tangent directions is
$\operatorname{Ric}_{g_4}(N,N)$.  The minus sign in this energy curvature
term is the one fixed in the definition of the index form.

\smallskip\noindent\textit{Reduction to a scalar divergence.}
Regard $a$ and $b$ as functions of $C$, put $t:=C^2$, and let primes denote
derivatives with respect to $C$.  Their defining formulas give
\[
 a+Cb=1,\quad a'+Cb'=-b,\quad b'=\frac{6a}s,
 \quad \alpha-\beta=-\frac{7-3t}s\,\ud C.
\]
Adding the three index expressions, with twice the mixed term, gives
an integrand whose coefficient of $|\nabla C|_g^2$ is
\[
\begin{aligned}
 &a'^2+b'^2+2Ca'b'-\frac{2ab'(7-3t)}s\\
 &\quad=b^2+(1-t)b'^2-\frac{2ab'(7-3t)}s
       =\frac{36t-48a^2}{s^2}.
\end{aligned}
\]
Let $\mathcal D$ be the resulting integrand with respect to $\ud A_g$
on the regular set.  The shape and Ricci formulas
\eqref{eq:shape-squared}--\eqref{eq:normal-ricci} and the angle-gradient
identity \eqref{eq:angle-gradient-C} give
\[
\begin{aligned}
 \mathcal D
 &=a^2\bigl(|B|_g^2-\operatorname{Ric}_{g_4}(N,N)\bigr)
      +\frac{36t-48a^2}{s^2}|\nabla C|_g^2\\
 &=a^2\left(\frac{9}{2}(1-t)^2+4-12t\right)
      +(9t-12a^2)(1-t)\\
 &=-\frac{s^2}{2}+\frac{-3+36t-9t^2}{s}
  =-q+\Delta_ga.
\end{aligned}
\]
Here we used $q=s^2/2$ from \eqref{eq:direct-step-16}.
For the last equality, direct differentiation gives
$a'=-12C/s^2$ and $a''=12(9t-1)/s^3$.  The angle equations
\eqref{eq:direct-step-14} and \eqref{eq:angle-gradient-C} then give
\[
 \Delta_ga
 =a'\Delta_gC+a''|\nabla C|_g^2
 =6t+\frac{3(9t-1)(1-t)}s
 =\frac{-3+36t-9t^2}{s}.
\]

All weighted-index expansions were taken globally with the smooth round
source metric.  On the regular set, two-dimensional conformal invariance
only rewrites their densities; it requires no integration by parts on a
punctured surface.  Moreover, $a$ is a smooth function on the entire round
sphere and
\[
 \Delta_ga\,\ud A_g
 =\Delta_{g_{\mathrm r}}a\,\ud A_{g_{\mathrm r}}
 \quad\text{on }\mathbb S^2\backslash\mathcal B.
\]
The right-hand density is smooth, and the omitted branch set is finite.
Its integral is therefore the integral of a Laplacian on the closed round
sphere, which is zero.  We conclude that
\begin{equation}
 \IF_\phi^{g_4}(W,W)
 =\int_{\mathbb S^2\backslash\mathcal B}
      (-q+\Delta_ga)\,\ud A_g
 =-\int_{\mathbb S^2\backslash\mathcal B}q\,\ud A_g.
 \label{eq:direct-step-24}
\end{equation}
All terms are integrable: the weighted-index densities have smooth round
representatives, and $q$ is bounded while the branch area density is
$|z|^{2m_p}$ times a smooth positive function.  This proves the identity.
\end{proof}

\begin{proof}[Proof of Proposition \ref{prop:direct-quantitative-instability}]
We combine the index identity with the total Jacobi mass and the uniform
length bound.  Let $W$ be the section supplied by Lemma
\ref{lem:canonical-field}.  Combining the mass formula in Lemma~\ref{lem:spherical-jacobi-metric}
with the direct index identity in Lemma~\ref{lem:canonical-index-identity} gives
\[
 \IF_\phi^{g_4}(W,W)
 =-\int_{\mathbb S^2\backslash\mathcal B}q\,\ud A_g
 \leq-8\pi.                                                        
\]
The length bound in Lemma \ref{lem:canonical-field} completes
the proof.
\end{proof}

\section{Contradiction to cone stability}\label{sec:cone-contradiction}

\begin{proof}[Proof of Theorem~\ref{thm:main}]
Recall that $\Spts=\operatorname{Sing}(\Q_0)$.  A point of this set
would give the lifted harmonic link $\widetilde\phi$ satisfying the
shifted stability inequality \eqref{eq:quarter-stability}.
Apply Proposition
\ref{prop:direct-quantitative-instability} to the nonconstant lifted
harmonic sphere $\widetilde\phi$ from
\eqref{eq:lifted-sphere-harmonic}.  Its test field $W$ satisfies
$\IF_{\widetilde\phi}^{g_4}(W,W)\leq-8\pi$ and
$|W|_{g_4}^2\leq13/4$.  Since the round source sphere has area $4\pi$,
\begin{equation}
 \IF_{\widetilde\phi}^{g_4}(W,W)
 +\frac14\int_{\mathbb S^2}|W|_{g_4}^2\,\ud A_{g_{\mathrm r}}
 \leq-8\pi+\frac{13\pi}{4}
 =-\frac{19\pi}{4}<0.
 \label{eq:final-negative}
\end{equation}
This contradicts the nonnegative left-hand side in
\eqref{eq:quarter-stability}.  Thus $\Spts=\varnothing$.
Interior smoothness at every regular point, as stated in
\cite[Theorem~I and the paragraph preceding Theorem~II]{SU82}, completes
the proof.
\end{proof}

\section*{Acknowledgments}

\subsection*{Use of LLM}

The research question was formulated by the authors, motivated by the
biaxial Landau--de Gennes model of Wang and Zhang \cite{WZ24} and by the
absence of point singularities for minimizing maps with a round
$\mathbb S^3$ target.  The authors' guiding idea was to use the stability
of minimizers by constructing a variation with negative second variation.
ChatGPT 5.6 sol identified the relation between the lifted energy and a
Berger metric and found the destabilizing vector field used in the proof.
LLM assistance was also used in preparing and revising the manuscript.
The manuscript was audited with a verification and revision agent developed
by the third author, using a workflow similar to that of
\href{https://github.com/BrookWW/BrookWW-ManuscriptAgent}{\texttt{BrookWW-ManuscriptAgent}}.
The authors take full responsibility for the verification and correctness
of all mathematical content, including the statements, proofs and
conclusions.

\subsection*{Funding}
This work was supported by Grant No.~2023YFA1008801 of the National Key
R\&D Program of China.

\end{document}